\documentclass[10pt]{article}

\usepackage{latexsym}

\usepackage{amsfonts}
\usepackage{amssymb}
\usepackage{amsmath}
\usepackage{amstext}
\usepackage{amsthm}

\newtheorem{thm}{Theorem}
\newtheorem{lem}[thm]{Lemma}

\theoremstyle{definition}

\newtheorem{ex}[thm]{Example}

\providecommand{\norm}[1]{\lVert#1\rVert}

\begin{document}

\title{\bf Finitely additive equivalent martingale measures}

\author{\textbf{Patrizia Berti}\\
Dipartimento di Matematica Pura ed Applicata ``G. Vitali''\\ Universit\`{a} di Modena e Reggio-Emilia \\ via Campi 213/B, 41100 Modena, Italy \\
e-mail: berti.patrizia@unimore.it \and
\textbf{Luca Pratelli}\\
 Accademia Navale, viale Italia 72, 57100 Livorno,
Italy\\
e-mail: pratel@mail.dm.unipi.it\and
\textbf{Pietro Rigo}\\
Dipartimento di Economia Politica e Metodi Quantitativi\\ Universit\`{a} di Pavia\\ via S. Felice 5, 27100 Pavia, Italy\\
e-mail: prigo@eco.unipv.it}

\maketitle

\begin{abstract}\noindent Let $L$ be a
linear space of real bounded random variables on the probability space $(\Omega,\mathcal{A},P_0)$. There is a finitely additive probability $P$ on $\mathcal{A}$, such that $P\sim P_0$ and $E_P(X)=0$ for all $X\in L$, if and only if $c\,E_Q(X)\leq\text{ess sup}(-X)$, $X\in L$, for some constant $c>0$ and (countably additive) probability $Q$ on $\mathcal{A}$ such that $Q\sim P_0$. A necessary condition for such a $P$ to exist is $\overline{L-L_\infty^+}\,\cap L_\infty^+=\{0\}$, where the closure is in the norm-topology. If $P_0$ is atomic, the condition is sufficient as well. In addition, there is a finitely additive probability $P$ on $\mathcal{A}$, such that $P\ll P_0$ and $E_P(X)=0$ for all $X\in L$, if and only if $\text{ess sup}(X)\geq 0$ for all $X\in L$.
\end{abstract}

\begin{quote}
{\footnotesize \noindent AMS 2010 {\em Subject Classification.} 60A05, 60A10, 28C05, 91B25, 91G10.
}
\end{quote}

\bigskip

\section{Introduction}\label{intro}

Throughout, $L$ is a
linear space of real bounded random variables on the probability space $(\Omega,\mathcal{A},P_0)$. The abbreviation "f.a.p." stands for {\em finitely additive probability}. Let $\mathbb{P}$ denote the set of f.a.p.'s on $\mathcal{A}$ and $\mathbb{P}_0\subset\mathbb{P}$ the subset of countably additive members of $\mathbb{P}$. In particular, $P_0\in\mathbb{P}_0$.

We aim to give conditions for the existence of $P\in\mathbb{P}$ such that
\begin{equation}\label{goal}
P\sim P_0\quad\text{and}\quad E_P(X)=0\text{ for each }X\in L,
\end{equation}
or such that
\begin{equation}\label{goal7}
P\ll P_0\quad\text{and}\quad E_P(X)=0\text{ for each }X\in L.
\end{equation}
As usual, $P\sim P_0$ means that $P$ and $P_0$ have the same null sets, while $P\ll P_0$ stands for $P(A)=0$ whenever $A\in\mathcal{A}$ and $P_0(A)=0$.

Under \eqref{goal} or \eqref{goal7}, $P$ is called an (equivalent or absolutely continuous) {\em martingale f.a.p.}. It is called an (equivalent or absolutely continuous) {\em martingale measure} in case $P\in\mathbb{P}_0$ and \eqref{goal} or \eqref{goal7} hold. The term "martingale" is motivated as follows. Let $\mathcal{F}=(\mathcal{F}_t:t\in T)$ be a filtration and
$S=(S_t:t\in T)$ a real $\mathcal{F}$-adapted process on $(\Omega,\mathcal{A},P_0)$, where $T\subset\mathbb{R}$ is any index set. Suppose $S_t$ a bounded random variable for each $t\in T$ and define
\begin{gather*}
L(\mathcal{F},S)=\,\text{Span}\,\{I_A\,(S_t-S_s):s,t\in T,\,s<t,\,A\in\mathcal{F}_s\}.
\end{gather*}
If $P\in\mathbb{P}_0$, then $S$ is a
$P$-martingale (with respect to $\mathcal{F}$) if and only if $E_P(X)=0$ for all $X\in L(\mathcal{F},S)$. If $P\in\mathbb{P}$ but $P\notin\mathbb{P}_0$, it looks natural to {\em define} $S$ a $P$-martingale in case $E_P(X)=0$ for all $X\in L(\mathcal{F},S)$. In this sense, a f.a.p. $P$ satisfying \eqref{goal} or \eqref{goal7} is of the martingale type.

Why to look for martingale f.a.p.'s ? We try to answer
this question by four (non independent) remarks.

\vspace{0.15cm}

{\bf (i)} Dating from de Finetti, the finitely additive theory of
probability is well founded and developed, even if not prevailing.
F.a.p.'s can be always extended to the power set and have a solid
motivation in terms of coherence; see Section \ref{v7bn5g}. Also, there are problems which can not be solved in the usual
countably additive setting, while admit a finitely additive
solution. Examples are in conditional probability,
convergence in distribution of non measurable random elements,
Bayesian statistics, stochastic integration and the first
digit problem. See e.g. \cite{BR04} and references therein. Moreover, in the finitely additive approach, one can clearly use
$\sigma$-additive evaluations. Merely, one is not obliged to do so.

\vspace{0.15cm}

{\bf (ii)} Martingale measures play a role in various financial frameworks. Their economic motivations, however, do not depend on whether they are $\sigma$-additive or not. See e.g. Chapter 1 of \cite{DS05}. In option
pricing, for instance, martingale f.a.p.'s give
free-arbitrage prices, precisely as their $\sigma$-additive
counterparts. Note also that many underlying ideas, in arbitrage price theory, were anticipated by de Finetti and Ramsey.

\vspace{0.15cm}

{\bf (iii)} It may be that conditions \eqref{goal} or \eqref{goal7} fail for each $P\in\mathbb{P}_0$ but hold for some $P\in\mathbb{P}$. This actually happens in some classical examples; see Examples \ref{chissase34} and \ref{trivsper54vhui8}. In addition, existence of martingale f.a.p.'s (both equivalent and absolutely continuous) can be given a simple characterization; see
Theorem \ref{b5tfc32w}.

\vspace{0.15cm}

{\bf (iv)} Investigating \eqref{goal}-\eqref{goal7} is natural from the functional analytic point of view. For instance, a necessary condition for the existence of equivalent martingale f.a.p.'s is $\overline{L-L_\infty^+}\,\cap L_\infty^+=\{0\}$, with the closure in the {\em norm-topology} of $L_\infty$. Such a condition is sufficient as well in case $P_0$ is atomic; see Theorem \ref{funcanapv89}. Recall that $\overline{L-L_\infty^+}\,\cap L_\infty^+=\{0\}$, with the closure in the {\em weak-star topology} of $L_\infty$, is necessary and sufficient for the existence of equivalent martingale measures; see \cite{K} and \cite{S}.

\vspace{0.15cm}

We next state our main result (Theorem \ref{b5tfc32w}). Define
\begin{gather*}
\text{ess sup}(X)=\inf\{a\in\mathbb{R}:P_0(X>a)=0\}=\inf\{\,\sup_AX:A\in\mathcal{A},\,P_0(A)=1\},
\\\norm{X}_\infty=\max\{\,\text{ess sup}(X),\,\text{ess sup}(-X)\},
\end{gather*}
for each essentially bounded random variable $X$. There is an equivalent martingale f.a.p. if and only if
\begin{equation}\label{dai}
\begin{array}{c}
\text{there are }\,Q\in\mathbb{P}_0,\,Q\sim P_0,\text{ and a constant }c>0\text{ such that
}\\
\vspace{-3mm}\\
c\,E_Q(X)\leq\text{ess sup}(-X)\quad\text{for each }X\in L.
\end{array}
\end{equation}
In addition, there is an absolutely continuous martingale f.a.p. if and only if
\begin{equation}\label{dai7}
\text{ess sup}(X)\geq 0\quad\text{for each }X\in L.
\end{equation}

Condition \eqref{dai7} has a transparent meaning. Even if more subtle, condition \eqref{dai} is essentially an internality constraint. In a suitable financial framework, the quantity $\text{ess sup}(-X)$ can be interpreted as "the maximum loss"; see Example 4.1 of \cite{D02}. Note also that, in testing whether \eqref{dai} holds, one can tentatively let $Q=P_0$.

Condition \eqref{dai} is automatically true in case
\begin{equation}\label{newxcond}
\text{ess sup}\,(X)\leq c^*\text{ess sup}(-X),\quad X\in L,\,\norm{X}_\infty=1,
\end{equation}
for some constant $c^*>0$. Suppose in fact \eqref{newxcond} holds and let $c=1/c^*$, $Q=P_0$ and $X\in L$. Condition \eqref{dai} trivially holds if $\norm{X}_\infty=0$. And, if $\norm{X}_\infty>0$, one obtains
\begin{gather*}
c\,E_{P_0}(X)\leq c\,\text{ess sup}(X)=c\,\norm{X}_\infty\,\text{ess sup}\bigl(\frac{X}{\norm{X}_\infty}\bigr)
\\\leq\norm{X}_\infty\,\text{ess sup}\bigl(\frac{-X}{\norm{X}_\infty}\bigr)
=\text{ess sup}(-X).
\end{gather*}

A further remark concerns the {\em no-arbitrage} condition
\begin{equation}\label{ob67ce}
P_0(X>0)>0\quad\Longleftrightarrow\quad P_0(X<0)>0\quad\text{for each
}X\in L.
\end{equation}
It turns out that
\begin{equation*}
\eqref{goal}\Longrightarrow\eqref{ob67ce}\Longrightarrow\eqref{goal7}\quad\text{while the converse implications are not true},
\end{equation*}
where \eqref{goal} and \eqref{goal7} are meant to hold for some $P\in\mathbb{P}$. In particular, no-arbitrage implies existence of an absolutely continuous martingale f.a.p. (but not necessarily of an absolutely continuous martingale measure; see Example \ref{chissase34}).

In fact, \eqref{goal} $\Rightarrow$ \eqref{ob67ce} follows from the representation
\begin{equation*}
P=\alpha\,P_1+(1-\alpha)\,Q\quad\text{where }\alpha\in [0,1),\,P_1\in\mathbb{P},\,Q\in\mathbb{P}_0\text{ and }Q\sim P_0,
\end{equation*}
which can be given to any equivalent martingale f.a.p. $P$; see Theorem \ref{b5tfc32w}. Since \eqref{ob67ce} trivially implies \eqref{dai7} and \eqref{dai7} $\Leftrightarrow$ \eqref{goal7}, then \eqref{ob67ce} $\Rightarrow$ \eqref{goal7}. Example \ref{chissase34} exhibits a situation where \eqref{ob67ce} holds and \eqref{goal} fails. And an example where \eqref{goal7} holds and \eqref{ob67ce} fails is $\Omega=\{1,2,\ldots\}$, $P_0\{\omega\}=2^{-\omega}$ for all $\omega\in\Omega$ and $L$ the linear span of $X(\omega)=\frac{1}{\omega}$.

A last note deals with the assumption that $L$ consists of {\em bounded} random variables. Even if strong, such an assumption can not be dropped. In fact, while de Finetti's coherence principle (our main tool) can be extended to unbounded random variables, the extensions are very far from granting an integral representation; see \cite{BR}, \cite{BRR01} and references therein.

\section{de Finetti's coherence principle}\label{v7bn5g}

Given any set $S$, let $\mathcal{P}(S)$ denote the power set of $S$ and $l^\infty(S)$ the collection of
real bounded functions on $S$. We write $E_P(X)=\int X\,dP$
whenever $X\in l^\infty(S)$ and $P$ is a f.a.p. on $\mathcal{P}(S)$. Since $X$
is the uniform limit of a sequence of simple functions, the integral
$\int X\,dP$ can be meant essentially in the usual sense; see
\cite{BHARAO} for details.

We briefly recall the notion of coherence. For more information, as
well as for some extensions (including conditional coherence), we
refer to \cite{BR}, \cite{BRR01}, \cite{R}, \cite{R87} and references therein.

Let $D\subset l^\infty(\Omega)$ and $E:D\rightarrow\mathbb{R}$.
According to de Finetti, $E$ is {\em coherent} in case
\begin{equation*}
\sup\sum_{i=1}^nc_i\,\bigl\{X_i-E(X_i)\bigr\}\geq 0
\end{equation*}
for all $n\geq 1$, $c_1,\ldots,c_n\in\mathbb{R}$ and
$X_1,\ldots,X_n\in D$. Heuristically, suppose $E$ describes your
previsions on members of $D$. If $E$ is coherent, it is impossible
to make you a sure looser, whatever $\omega\in\Omega$ turns out to
be true, by some finite combinations of bets (on $X_1,\ldots,X_n$
with stakes $c_1,\ldots,c_n$).

Say that $E$ is {\em internal} in case $\inf X\leq E(X)\leq\sup X$
for all $X\in D$. If $D$ is a linear space, $E$ is coherent if and
only if it is linear and internal. Also, internality
reduces to $E(X)\leq\sup X$ for all $X\in D$, provided $D$ is a
linear space and $E$ a linear functional.

A coherent map $E$ can be coherently extended to $l^\infty(\Omega)$.
Suppose in fact $E$ is coherent. It is not hard to see that $E$ can
be extended to a linear internal functional $E_1$ on the linear
space spanned by $D$. In turn, by Hahn-Banach theorem, $E_1$ can be
extended to a linear internal functional $E_2$ on
$l^\infty(\Omega)$. Note that, letting $P(A)=E_2(I_A)$ for all
$A\subset\Omega$, one obtains a f.a.p. $P$ on $\mathcal{P}(\Omega)$. As simple
functions are dense in $l^\infty(\Omega)$ under the sup-norm, one
also obtains $E_2(X)=\int X\,dP$ for all $X\in l^\infty(\Omega)$.
Thus, $E:D\rightarrow\mathbb{R}$ is coherent if and only if
\begin{equation*}
E(X)=\int X\,dP=E_P(X),\quad X\in D,
\end{equation*}
for some f.a.p. $P$ on $\mathcal{P}(\Omega)$. This is, according to us, the
more transparent way of thinking of coherence.

We next give a couple of lemmas. The first is essentially known (see Section 10.3 of \cite{BHARAO}) but we prove it to keep the paper self-contained. Say that $P\in\mathbb{P}$ is {\em pure} in case it does not have a non trivial
$\sigma$-additive part, that is
\begin{equation*}
\text{if }\Gamma\text{ is a }\sigma\text{-additive measure on
}\mathcal{A}\text{ and }0\leq \Gamma\leq P,\text{ then }\Gamma=0.
\end{equation*}
By a result of Yosida-Hewitt, any $P\in\mathbb{P}$ can be written as
\begin{equation*}
P=\alpha\,P_1+(1-\alpha)\,Q
\end{equation*}
where $\alpha\in [0,1]$, $P_1\in\mathbb{P}$ is pure (unless $\alpha=0$) and $Q\in\mathbb{P}_0$.

\begin{lem}\label{pure}
Let $P\in\mathbb{P}$ be such that $P\ll P_0$. Then, $P$ is pure if and only if there is a countable partition $H_1,H_2,\ldots$ of $\Omega$ such
that $H_n\in\mathcal{A}$ and $P(H_n)=0$ for all $n$.
\end{lem}

\begin{proof}
The "if" part is trivial. Suppose $P$ is pure. It suffices to
prove that, given $\epsilon>0$, there is $A\in\mathcal{A}$ with
$P(A)=0$ and $P_0(A^c)<\epsilon$. In this case, in fact, there is
an increasing sequence $A_1\subset A_2\subset\ldots$ satisfying $A_n\in\mathcal{A}$, $P(A_n)=0$ and $P_0(A_n^c)<1\,/\,n$ for
all $n\geq 1$. Let $B=(\cup_n A_n)^c$, $H_1=A_1\cup B$, and
$H_n=A_n\setminus A_{n-1}$ for $n>1$. Then, $H_1,H_2,\ldots$ is a
partition of $\Omega$ in $\mathcal{A}$ and $P(H_n)=0$
for $n>1$. Also, $P(H_1)=P(B)=0$ since $P\ll P_0$ and $P_0(B)=0$. Next, fix $\epsilon>0$, and define
\begin{equation*}
\Gamma(A)=\inf\{P(B)+P_0(A\setminus B):B\in\mathcal{A},\,B\subset
A\},\quad A\in\mathcal{A}.
\end{equation*}
It is straightforward to check that $\Gamma$ is a finitely additive
measure on $\mathcal{A}$. Since $P_0\in\mathbb{P}_0$ and $0\leq\Gamma\leq
P_0$, then $\Gamma$ is $\sigma$-additive. Since $P$ is pure and $0\leq\Gamma\leq P$, then $\Gamma=0$. Hence, for each $n\geq 1$, there is
$B_n\in\mathcal{A}$ satisfying
$P(B_n)+P_0(B_n^c)<\epsilon\,/\,2^n$. Let $A=\cap_nB_n$. Then, $A\in\mathcal{A}$, $P(A)=0$ and $P_0(A^c)\leq\sum_nP_0(B_n^c)<\epsilon$.
\end{proof}

The second lemma is fundamental for our main results. It is connected to Lemma 1 of \cite{HS}.

\begin{lem}\label{kumon} Let $D\subset l^\infty(\Omega)$ be a linear space,
$E:D\rightarrow\mathbb{R}$ a linear functional, and $\mathcal{E}$ a
class of subsets of $\Omega$ such that $A\cap B\in\mathcal{E}$
whenever $A,\,B\in\mathcal{E}$. There is a f.a.p. $P$ on $\mathcal{P}(\Omega)$
satisfying
\begin{equation*}
E(X)=E_P(X)\quad\text{and}\quad P(A)=1\quad\text{for all }X\in
D\text{ and }A\in\mathcal{E}
\end{equation*}
if and only if
\begin{equation}\label{ng5rt}
\sup_AX\geq E(X)\quad\text{for all }X\in D\text{ and
}A\in\mathcal{E}.
\end{equation}
\end{lem}
\begin{proof}
The "only if" part is trivial. Suppose \eqref{ng5rt} holds and fix
$A\in\mathcal{E}$. If $X,\,Y\in D$ and $X=Y$ on $A$, then
\eqref{ng5rt} implies $E(X)=E(Y)$. Hence, one can define
$\phi(X|A)=E(X)$ for $X\in D$, where $X|A$ denotes the restriction
of $X$ to $A$. By \eqref{ng5rt}, $\phi$ is a coherent map on
$\{X|A:X\in D\}$. Take a f.a.p. $P_1$ on $\mathcal{P}(A)$ satisfying
$\phi(X|A)=\int(X|A)\,dP_1$, $X\in D$, and define
$P(B)=P_1(A\cap B)$ for $B\subset\Omega$. Then,
$P$ is a f.a.p. on $\mathcal{P}(\Omega)$ such that $P(A)=1$ and
$E(X)=E_P(X)$ for all $X\in D$. Next, let $\mathcal{Z}$ be the
class of $[0,1]$-valued functions on $\mathcal{P}(\Omega)$. When equipped with the product topology, $\mathcal{Z}$ is
compact and
\begin{equation*}
F_A=\{P\in\mathcal{Z}:P\text{ is a f.a.p.,
}P(A)=1,\,E(X)=E_P(X)\text{ for all }X\in D\}
\end{equation*}
is closed for each $A\subset\Omega$. By what already proved,
$F_A\neq\emptyset$ for all $A\in\mathcal{E}$. Hence, since $\mathcal{E}$ is
closed under finite intersections, $\{F_A:A\in\mathcal{E}\}$
has the finite intersection property. It follows that
$\bigcap_{A\in\mathcal{E}}F_A\neq\emptyset$, and this concludes the
proof.
\end{proof}

\section{Equivalent and absolutely continuous martingale f.a.p.'s}

\begin{thm}\label{b5tfc32w}
\hspace{0.5cm}\\
\vspace{-0.25cm}
\begin{itemize}
\item[{\bf (a)}] There is $P\in\mathbb{P}$ such that $P\ll P_0$ and
$E_P(X)=0$, $X\in L$, if and only if $\,\text{ess sup}(X)\geq 0$ for each $X\in L$.
\item[{\bf (b)}] There is $P\in\mathbb{P}$ such that $P\sim P_0$ and
$E_P(X)=0$, $X\in L$, if and only if condition \eqref{dai} holds.
\end{itemize}
Moreover, every equivalent martingale f.a.p. $P$ admits the representation \linebreak $P=\alpha\,P_1+(1-\alpha)\,Q$ where $\alpha\in [0,1)$, $P_1\in\mathbb{P}$ is pure (unless $\alpha=0$), $Q\in\mathbb{P}_0$ and $Q\sim P_0$.
\end{thm}
\begin{proof} Let $E:L\rightarrow\mathbb{R}$ and $\mathcal{E}=\{A\in\mathcal{A}:P_0(A)=1\}$. Since the elements of $L$ are $\mathcal{A}$-measurable, there is $P\in\mathbb{P}$ such that $P\ll P_0$  and $E_P(X)=E(X)$ for $X\in L$ if and only if there is a f.a.p. $T$ on $\mathcal{P}(\Omega)$ such that $T(A)=1$ and $E_T(X)=E(X)$ for $A\in\mathcal{E}$ and $X\in L$. Hence, part (a) follows from Lemma \ref{kumon} applied with $D=L$ and $E=0$. As to part (b), suppose condition \eqref{dai} holds and define $D=L$ and $E(X)=-c\,E_Q(X)$ for $X\in L$. By \eqref{dai},
\begin{equation*}
E(X)=c\,E_Q(-X)\leq\text{ess sup}(X)\leq\sup_AX\quad\text{for all }X\in L\text{ and }A\in\mathcal{E}.
\end{equation*}
By Lemma \ref{kumon},
there is $P_1\in\mathbb{P}$ such that $P_1\ll P_0$ and
$E_{P_1}(X)=-c\,E_Q(X)$ for $X\in L$. Hence, an equivalent martingale f.a.p. is
\begin{equation*}
P=\frac{P_1+c\,Q}{1+c}.
\end{equation*}

Next, let $P\in\mathbb{P}$ be such that $P\sim P_0$ and $E_P(X)=0$ for $X\in L$. To get condition \eqref{dai}, it suffices to show that $P=\alpha\,P_1+(1-\alpha)\,Q$ where $\alpha\in [0,1)$, $P_1\in\mathbb{P}$ is pure (unless $\alpha=0$), $Q\in\mathbb{P}_0$ and $Q\sim P_0$. In fact, suppose that $P$ can be written in this way. If $\alpha=0$, letting $c=1$ and $Q=P$, one trivially obtains
\begin{equation*}
c\,E_Q(X)=E_P(X)=0=E_P(-X)\leq\text{ess sup}(-X)\quad\text{for all }X\in L.
\end{equation*}
If $\alpha\in (0,1)$, then $P_1\ll P_0$ so that $E_{P_1}(X)\leq\text{ess sup}(X)$ for each bounded random variable $X$. Let $c=\frac{1-\alpha}{\alpha}$ and $X\in L$. Since $E_P(X)=0$, one again obtains
\begin{gather*}
c\,E_Q(X)=\frac{1-\alpha}{\alpha}\,\,\frac{E_P(X)-\alpha\,E_{P_1}(X)}{1-\alpha}=-E_{P_1}(X)=E_{P_1}(-X)\leq\text{ess sup}(-X).
\end{gather*}

We finally prove that $P$ admits the desired representation. By Yosida-Hewitt's theorem, $P=\alpha\,P_1+(1-\alpha)\,Q$ with $\alpha\in[0,1]$, $P_1\in\mathbb{P}$ pure (unless $\alpha=0$) and $Q\in\mathbb{P}_0$. Let
$H_1,H_2,\ldots$ be a countable partition of
$\Omega$ in $\mathcal{A}$. Since $P\sim P_0$, it must be
$P(H_n)>0$ for some $n$. By Lemma \ref{pure}, $P$ is not pure. Hence $\alpha<1$, and this in turn implies $Q\ll P_0$. It remains to show that $Q\sim P_0$. If $\alpha=0$, then $Q=P\sim P_0$. Suppose $\alpha\in (0,1)$. Let $A=\{f=0\}$ where $f$ is a density of $Q$ with respect to $P_0$. If
$P_1(A)=0$, then $P(A)=(1-\alpha)\,Q(A)=0$, so that $P_0(A)=0$. Thus, it can be
assumed $P_1(A)>0$. Toward a contradiction, suppose
also that $P_0(A)>0$. In that case, since $P\sim P_0$ and $Q(A)=0$,
one obtains
\begin{equation*}
P_1(\cdot\mid
A)\sim P_0(\cdot\mid A).
\end{equation*}
Hence, if $H_1,H_2,\ldots$ is as above, it must be $P_1(H_n)\geq P_1(A)\,P_1(H_n\mid
A)>0$ for some $n$. Since $P_1\ll P_0$, Lemma \ref{pure} implies that $P_1$ is not pure, and this is a contradiction. Therefore $P_0(A)=0$, that is, $Q\sim P_0$. This concludes the proof.

\end{proof}

Theorem \ref{b5tfc32w} is our main result. To stress its possible role, we discuss a few (classical) examples. Recall (from Section \ref{intro}) that \eqref{newxcond} $\Rightarrow$ \eqref{dai}.

\begin{ex} {\bf (Finite state space)}\label{finnumat65}
If $\Omega$ is finite, no-arbitrage implies existence of equivalent martingale measures. This {\em well known} fact follows trivially from Theorem \ref{b5tfc32w}. Indeed, when $\Omega$ is finite, $\mathbb{P}=\mathbb{P}_0$ and $L$ is finite-dimensional. Thus, by Theorem \ref{b5tfc32w}, it suffices to show that \eqref{ob67ce} $\Rightarrow$ \eqref{newxcond} if $L$ is finite-dimensional. Regard $L$ as a subspace of $L_\infty$, where $L_\infty=L_\infty(\Omega,\mathcal{A},P_0)$ is equipped with the norm-topology. Define $K=\{X\in L:\norm{X}_\infty=1\}$ and
\begin{equation*}
\phi(X)=\frac{\text{ess sup}(X)}{\text{ess sup}(-X)}\quad\text{for all }X\in K.
\end{equation*}
By \eqref{ob67ce}, $\phi:K\rightarrow (0,\infty)$ is well defined and continuous. Since $L$ is finite-dimensional, it is not hard to see that $K$ is compact. Thus, condition \eqref{newxcond} holds.

\end{ex}

A $P_0$-{\em atom} is a set $A\in\mathcal{A}$ with $P_0(A)>0$ and $P_0(\cdot\mid A)\in\{0,1\}$, and $P_0$ is {\em atomic} if there is a countable partition $A_1,A_2,\ldots$ of $\Omega$ such that $A_n$ is a $P_0$-atom for all $n$.

\begin{ex} {\bf (Atomic $P_0$)}\label{opl98v5z3mk8}
As in Example \ref{finnumat65}, sometimes, Theorem \ref{b5tfc32w} helps in proving existence of {\em equivalent martingale measures}. Suppose $P_0$ atomic and fix a partition $A_1,A_2,\ldots$ of $\Omega$ with each $A_n$ a $P_0$-atom. If the $A_n$ are finitely many, we are essentially in the framework of Example \ref{finnumat65}. Thus, suppose the $A_n$ are infinitely many. Then, there is an equivalent martingale measure if condition \eqref{newxcond} holds and
\begin{equation}\label{c0ijhy}
\lim_nX|A_n=0\quad\text{for all }X\in L.
\end{equation}
Here, $X|A_n$ denotes the a.s.-constant value of $X$ on $A_n$. In fact, by \eqref{newxcond} and Theorem \ref{b5tfc32w}, there is an equivalent martingale f.a.p. $P$. Write $P=\alpha\,P_1+(1-\alpha)\,Q$, where $\alpha\in [0,1)$, $P_1\in\mathbb{P}$ is pure (unless $\alpha=0$), $Q\in\mathbb{P}_0$ and $Q\sim P_0$. If $\alpha=0$, then $Q=P$ is an equivalent martingale measure. Let $\alpha>0$. If $P_1(A_n)>0$, since $P_1\ll P_0$ and $P_0(\cdot\mid A_n)$ is 0-1 valued, one obtains $P_1(\cdot\mid A_n)=P_0(\cdot\mid A_n)$. But this is a contradiction, for $P_1(\cdot\mid A_n)$ is pure. Hence, $P_1(A_n)=0$ for all $n$, and condition \eqref{c0ijhy} implies $E_{P_1}(X)=0$ for all $X\in L$. Therefore, $Q$ is again an equivalent martingale measure. Finally, suppose \eqref{c0ijhy} fails. Then, there is an equivalent martingale measure provided condition \eqref{newxcond} is turned into
\begin{equation}\label{primaopoi98vc}
\text{ess sup}\,(X\,Y)\leq c^*\text{ess sup}(-X\,Y),\quad X\in L,\,\norm{X\,Y}_\infty=1,\tag{5*}
\end{equation}
for some constant $c^*>0$ and bounded random variable $Y$ such that $Y>0$ and $\lim_nY|A_n=0$. In fact, $\lim_nXY|A_n=0$ for all $X\in L$, so that condition \eqref{primaopoi98vc} yields $E_Q(X\,Y)=0$, $X\in L$, for some $Q\in\mathbb{P}_0$, $Q\sim P_0$. Hence, an equivalent martingale measure is $Q^*(A)=E_Q\bigl(Y\,I_A\bigr)/E_Q(Y)$, $A\in\mathcal{A}$.

\end{ex}

\begin{ex}\label{chissase34} {\bf (An example from \cite{DMW} revisited)} Let $Y_n:\Omega\rightarrow\{-1,1\}$, $n\geq 1$, and $\mathcal{A}=\sigma(Y_1,Y_2,\ldots)$. Suppose $(Y_n)$ i.i.d. under $P_0$ with $0<P_0(Y_1=1)<1/2$. Also, suppose $(Y_n)$ i.i.d. under $Q_0$, where $Q_0\in \mathbb{P}_0$, with $Q_0(Y_1=1)=1/2$. Define
\begin{equation*}
S_0=0,\quad S_n=\sum_{i=1}^nY_i,\quad \mathcal{F}_n=\sigma(S_0,S_1,\ldots,S_n),\quad L=L(\mathcal{F},S),
\end{equation*}
where $L(\mathcal{F},S)$ has been defined in Section \ref{intro}.

Since $P_0\bigl(Y_1=y_1,\ldots,Y_n=y_n\bigr)>0$ for all $n\geq 1$ and $y_1,\ldots,y_n\in\{-1,1\}$, there is no-arbitrage, i.e., condition \eqref{ob67ce} holds. In addition, if $P\in\mathbb{P}$ is such that $E_P(X)=0$ for $X\in L$, then $P=Q_0$ on $\cup_n\mathcal{F}_n$. In fact, $E_P(I_A\,Y_n)=0$ yields $P\bigl(A\cap\{Y_n=1\}\bigr)=P\bigl(A\cap\{Y_n=-1\}\bigr)$ for each $n\geq 1$ and $A\in\mathcal{F}_{n-1}$. Thus, $(Y_n)$ is i.i.d. under $P$ with $P(Y_1=1)=1/2$.

Since $Q_0$ is the only member of $\mathbb{P}_0$ which makes $(S_n)$ a martingale and
\begin{equation*}
\frac{S_n}{n}\overset{Q_0-a.s.}\longrightarrow 0\quad\text{while}\quad\frac{S_n}{n}\overset{P_0-a.s.}\longrightarrow E_{P_0}(Y_1)< 0,
\end{equation*}
there are not absolutely continuous martingale measures. Instead, by Theorem \ref{b5tfc32w} and no-arbitrage, there are absolutely continuous martingale f.a.p.'s. Finally, no equivalent martingale f.a.p. is available. Suppose in fact $P$ is an equivalent martingale f.a.p.. By Theorem \ref{b5tfc32w} and since $E_P(X)=0$ for $X\in L$,
\begin{equation*}
Q_0=P\geq (1-\alpha)\,Q\quad\text{on }\cup_n\mathcal{F}_n
\end{equation*}
for some $\alpha\in [0,1)$ and $Q\in\mathbb{P}_0$ such that $Q\sim P_0$. Since $Q,\,Q_0\in\mathbb{P}_0$ and $\cup_n\mathcal{F}_n$ is a field, it follows that $Q_0\geq (1-\alpha)\,Q$ on $\sigma\bigl(\cup_n\mathcal{F}_n\bigr)=\mathcal{A}$. On noting that $\alpha<1$, one obtains the contradiction $P_0\sim Q\ll Q_0$.
\end{ex}

Equivalent martingale f.a.p.'s may be available even if equivalent martingale measures fail to exist. As a trivial example, take a pure f.a.p. $P_1$ such that $P_1\ll P_0$ (such a $P_1$ exists for several choices of $P_0$). Define $P=\frac{P_0+P_1}{2}$ and
\begin{equation*}
L=\{X:X\text{ bounded random variable, }E_P(X)=0\}.
\end{equation*}
Then, $P\sim P_0$. But for each $Q\in\mathbb{P}_0$, since $Q\neq P$, one obtains $E_Q(X)\neq 0$ for some $X\in L$. Here is a less trivial example.

\begin{ex}\label{trivsper54vhui8} {\bf (An example from \cite{BP} revisited)}
Let $\Omega=\{1,2,\ldots\}$, $\mathcal{A}=\mathcal{P}(\Omega)$, and $P_0\{\omega\}=2^{-\omega}$ for all $\omega\in\Omega$. For each $n\geq 0$, define $A_n=\{n+1,n+2,\ldots\}$. Define also $L=L(\mathcal{F},S)$, where
\begin{gather*}
\mathcal{F}_0=\{\emptyset,\Omega\},\quad\mathcal{F}_n=\sigma\bigl(\{1\},\ldots,\{n\}\bigr),\quad S_0=1,\quad\text{and}
\\ S_n(\omega)=\frac{1}{2^n}\,I_{A_n}(\omega)+\frac{\omega^2+2\omega+2}{2^\omega}\,(1-I_{A_n}(\omega))\quad\text{for all }\omega\in\Omega.
\end{gather*}
As shown in \cite{BP}, no $Q\in\mathbb{P}_0$ satisfies $E_Q(X)=0$ for all $X\in L$. However, equivalent martingale f.a.p.'s are available. Define in fact $Q\{\omega\}=\frac{1}{\omega}-\frac{1}{\omega+1}$ for all $\omega\in\Omega$. Then, $Q\in\mathbb{P}_0$ and $Q\sim P_0$. Since $S_{n+1}=S_n$ on $A_n^c$, each $X\in L(\mathcal{F},S)$ can be written as
\begin{gather*}
X=\sum_{j=0}^kb_j\,I_{A_j}\,(S_{j+1}-S_j)
\end{gather*}
for some $k\geq 0$ and $b_0,\ldots,b_k\in\mathbb{R}$. On noting that $X=-\sum_{j=0}^k\frac{b_j}{2^{j+1}}$ on $A_{k+1}$, one obtains
\begin{gather*}
E_Q(X)=\sum_{j=0}^k\frac{b_j}{2^{j+1}}\,\big\{\bigl((j+1)^2+2(j+1)\bigr)\,Q\{j+1\}-Q(A_{j+1})\bigr\}
\\=\sum_{j=0}^k\frac{b_j}{2^{j+1}}\,\big\{\frac{(j+1)^2+2(j+1)}{(j+1)(j+2)}-\frac{1}{(j+2)}\bigr\}
\\=\sum_{j=0}^k\frac{b_j}{2^{j+1}}\leq\sup\,(-X)=\text{ess sup}(-X).
\end{gather*}
Therefore, condition \eqref{dai} holds (with $c=1$) and Theorem \ref{b5tfc32w} grants the existence of an equivalent martingale f.a.p. $P$. Incidentally, such a $P$ can be taken of the form $P=\frac{Q+P_1}{2}$, where $Q$ is as above and $P_1\in\mathbb{P}$ is such that $P_1(A_n)=1$ for all $n$. Note also that condition \eqref{newxcond} fails in this example.
\end{ex}

Finally, we take a functional analytic point of view and we investigate the connections between existence of equivalent martingale f.a.p.'s and measures. Write $U-V=\{u-v:u\in U,\,v\in V\}$ whenever $U,\,V$ are subsets of a linear space. Let $L_p=L_p(\Omega,\mathcal{A},P_0)$ for all $p\in [1,\infty]$. We regard $L$ as a subspace of $L_\infty$ and we let $L_\infty^+=\{X\in L_\infty:X\geq 0\}$. Since $L_\infty$ is the dual of $L_1$, it can be equipped with the weak-star topology $\sigma(L_\infty,L_1)$. Thus, $\sigma(L_\infty,L_1)$ is the topology on $L_\infty$ generated by the linear functionals $X\mapsto E_{P_0}\bigl(X\,Y)$ for $Y\in L_1$.

A classical result of Kreps \cite{K} (see also \cite{S}) states that existence of equivalent martingale measures amounts to
\begin{equation*}
\overline{L-L_\infty^+}\,\cap L_\infty^+=\{0\}\quad\text{with the closure in }\sigma(L_\infty,L_1).
\end{equation*}
A (natural) question is what happens if the closure is taken in the norm-topology.

\begin{thm}\label{funcanapv89}
There is an equivalent martingale f.a.p. if and only if
\begin{equation}\label{bazg78j}
\begin{array}{c}
L_\infty^+\subset U\cup\{0\}\quad\text{and}\quad (L-L_\infty^+)\cap U=\emptyset\\
\vspace{-3mm}\\
\text{for some norm-open convex set }U\subset L_\infty.
\end{array}
\end{equation}
In particular, a necessary condition for the existence of an equivalent martingale f.a.p. is
\begin{equation}\label{margr}
\overline{L-L_\infty^+}\,\cap L_\infty^+=\{0\}\quad\text{with the closure in the norm-topology}.
\end{equation}
If $P_0$ is atomic, condition \eqref{margr} is sufficient as well.
\end{thm}

\begin{proof} Let $L_\infty$ be equipped with the norm-topology and
\begin{equation*}
V=L-L_\infty^+,\quad W=\overline{V}=\overline{L-L_\infty^+},\quad\mathbb{M}=\{P\in\mathbb{P}:P\ll P_0,\,E_P(X)=0\text{ for }X\in L\}.
\end{equation*}

Suppose $P$ is an equivalent martingale f.a.p. and define
\begin{equation*}
U=\{X\in L_\infty:E_P(X)>0\}.
\end{equation*}
Since $E_P(X-Y)=-E_P(Y)\leq 0$ whenever $X\in L$ and $Y\in L_\infty^+$, then $U\cap V=\emptyset$. For each $X\in L_\infty^+$, $X\neq 0$, there is $\epsilon>0$ with $P_0(X\geq\epsilon)>0$, so that
\begin{equation*}
E_P(X)=E_P(X^+)\geq\epsilon\,P(X\geq\epsilon)>0.
\end{equation*}
Hence, $L_\infty^+\subset U\cup\{0\}$. Further, $U$ is open and convex for the map $X\mapsto E_P(X)$ is linear and continuous. Conversely, suppose condition \eqref{bazg78j} holds. Since $V$ is convex, $U$ is open convex and $U\cap V=\emptyset$, there is a linear (continuous) functional $f:L_\infty\rightarrow\mathbb{R}$ such that $f(X)>f(Y)$ for all $X\in U$ and $Y\in V$. Since $f>0$ on $U$ (due to $0\in V$) then $f$ is positive. Since $f(1)>0$ (for $1\in U$) it can be assumed $f(1)=1$. By Lemma \ref{kumon}, $f(X)=E_P(X)$, $X\in L_\infty$, for some $P\in\mathbb{P}$ with $P\ll P_0$. Such a $P$ is an equivalent martingale f.a.p.. In fact, since $L$ is a linear space and $\sup_Lf\leq\sup_Vf<\infty$, then $f=0$ on $L$. Thus, $P\in\mathbb{M}$. And for each $A\in\mathcal{A}$ with $P_0(A)>0$, one obtains $P(A)=f(I_A)>0$ for $I_A\in U$.

Next, under condition \eqref{bazg78j}, $W=\overline{V}\subset U^c$. Hence, it is obvious that \eqref{bazg78j} $\Rightarrow$ \eqref{margr}.

Finally, suppose that \eqref{margr} holds. Fix $A\in\mathcal{A}$ with $P_0(A)>0$. Since $I_A\notin W$, one obtains $f(I_A)>\sup_Wf$ for some linear (continuous) functional $f:L_\infty\rightarrow\mathbb{R}$. Given $X\in L_\infty^+$, since $-n\,X\in W$ for all $n\geq 1$, it follows that
\begin{equation*}
\sup_n\,[-n\,f(X)]=\sup_nf(-n\,X)\leq\sup_Wf<\infty.
\end{equation*}
Hence $f(X)\geq 0$, i.e., $f$ is positive. Since $f(1)>0$ (otherwise, $f$ is identically null) it can be assumed $f(1)=1$. Again, by Lemma \ref{kumon}, $f(X)=E_P(X)$, $X\in L_\infty$, for some $P\in\mathbb{P}$ with $P\ll P_0$. Since $0\in W$, then $P(A)=f(I_A)>0$. Since $L$ is a linear space and $\sup_Lf\leq\sup_Wf<\infty$, then $E_P(X)=0$ for all $X\in L$. Summarizing, for each $A\in\mathcal{A}$ with $P_0(A)>0$, there is $P_A\in\mathbb{M}$ such that $P_A(A)>0$. If $P_0$ is atomic, as we now assume, there is a partition $A_1,A_2,\ldots$ of $\Omega$ such that $A_n\in\mathcal{A}$, $P_0(A_n)>0$ and $P_0(\cdot\mid A_n)$ is 0-1 valued for all $n$. Define $P=\sum_{n=1}^\infty 2^{-n}P_{A_n}$. Then, $P\in\mathbb{M}$. For each $A\in\mathcal{A}$ with $P_0(A)>0$, one obtains $P_0(A\cap A_n)>0$ and $P_0(A^c\cap A_n)=0$ for some $n$. On noting that $P_{A_n}\ll P_0$,
\begin{equation*}
2^nP(A)\geq P_{A_n}(A\cap A_n)=P_{A_n}(A_n)-P_{A_n}(A^c\cap A_n)=P_{A_n}(A_n)>0.
\end{equation*}
Therefore, $P$ is an equivalent martingale f.a.p..
\end{proof}

For general $P_0$, we do not know whether condition \eqref{margr} suffices for the existence of equivalent martingale f.a.p.'s. Another open problem is whether condition \eqref{newxcond} implies the existence of equivalent martingale measures.

Finally, we note that condition \eqref{margr} looks like the {\em no free lunch with vanishing risk} condition of \cite{DS94}. Actually, the only (but basic) difference pertains the form of $L$. In \cite{DS94}, $L$ is a certain class of stochastic integrals (in a fixed time interval and driven by a fixed semi-martingale) while in this paper $L$ is any subspace of $L_\infty$. We refer to \cite{KAB} for a concise and elegant proof of the results in \cite{DS94}.


\begin{thebibliography}{99}

\bibitem{BP} Back K., Pliska S.R. (1991)
On the fundamental theorem of asset pricing with an infinite state space, {\em J. Math. Econ.}, 20, 1-18.

\bibitem{BR} Berti P., Rigo P. (2000)
Integral representation of linear functionals on spaces of unbounded
functions, {\em Proc. Amer. Math. Soc.}, 128, 3251-3258.

\bibitem{BRR01} Berti P., Regazzini E., Rigo P. (2001)
Strong previsions of random elements, {\em Statist. Methods Appl.}, 10, 11-28.

\bibitem{BR04} Berti P., Rigo P. (2004) Convergence in distribution of non
measurable random elements, {\em Ann. Probab.}, 32, 365-379.

\bibitem{BHARAO} Bhaskara Rao K.P.S., Bhaskara Rao M. (1983) {\em Theory of charges}, Academic Press.

\bibitem{DMW} Dalang R., Morton A., Willinger W. (1990) Equivalent martingale measures and no-arbitrage in
stochastic securities market models, {\em Stoch. and Stoch. Reports}, 29, 185-201.

\bibitem{D02} Delbaen F. (2002) Coherent risk measures on general probability spaces, {\em Advances in Finance and Stochastics:
essays in honour of Dieter Sondermann}, Springer, 1-38.

\bibitem{DS94} Delbaen F., Schachermayer W. (1994) A general version of the fundamental theorem of asset pricing, {\em Math. Annalen}, 300, 463-520.

\bibitem{DS05} Delbaen F., Schachermayer W. (2006) {\em The mathematics of arbitrage}, Springer.

\bibitem{HS} Heath D., Sudderth W.D. (1978) On finitely additive priors, coherence and extended admissibility, {\em Ann. Statist.}, 6, 333-345.

\bibitem{KAB} Kabanov Y.M. (1997) On the FTAP of Kreps-Delbaen-Schachermayer, {\em Y.M. Kabanov (ed.) et al., Statistics
and control of stochastic processes}, the Lipster Festschrift, Papers from Steklov seminar held
in Moscow, 1995-1996, World Scientific, 191-203.

\bibitem{K} Kreps D.M. (1981) Arbitrage and equilibrium in economics with infinitely many commodities, {\em J. Math. Econ.}, 8, 15-35.

\bibitem{R} Regazzini E. (1985) Finitely additive conditional probabilities, {\em Rend. Sem. Mat. Fis.
Milano}, 55, 69-89. Corrections in: {\em Rend. Sem. Mat. Fis.
Milano} (1987) 57, 599.

\bibitem{R87} Regazzini E. (1987) de Finetti's coherence and statistical inference, {\em Ann. Statist.}, 15, 845-864.

\bibitem{S} Stricker C. (1990) Arbitrage et lois de martingale, {\em Ann. Inst. Henri Poincar\'e $-$Probab. et Statist.}, 26, 451-460.

\end{thebibliography}
\end{document}